\documentclass[11pt]{amsart}

\usepackage{latexsym, amssymb,enumerate}
\usepackage[T1]{fontenc}
\usepackage{textcomp}
\usepackage{appendix}

\newcommand{\be}{\begin{equation}}
\newcommand{\ee}{\end{equation}}
\newcommand{\beq}{\begin{eqnarray}}
\newcommand{\eeq}{\end{eqnarray}}

\def\H{{\Bbb H}}

\def\R{{\mathfrak R}}

\newtheorem{prop}{Proposition}[section]
\newtheorem{theo}[prop]{Theorem}
\newtheorem{lemm}[prop]{Lemma}

\def\begeq{\begin{equation}}
\def\endeq{\end{equation}}

\def\R{\Bbb R}

\def\s{\sigma}

\def\odot{\setbox0=\hbox{$\bigcirc$}\relax \mathbin {\hbox
to0pt{\raise.5pt\hbox to\wd0{\hfil $\wedge$\hfil}\hss}\box0 }}

\numberwithin{equation} {section}

\begin{document}
\author{Jie Wu}
\address{School of Mathematical Sciences, University of Science and Technology
of China Hefei 230026, P. R. China
\and
 Albert-Ludwigs-Universit\"at Freiburg,
Mathematisches Institut
Eckerstr. 1
D-79104 Freiburg
}
\email{jie.wu@math.uni-freiburg.de}
\thanks{The author  is partly supported by SFB/TR71
``Geometric partial differential equations''  of DFG}
\begin{abstract}
This paper gives a new characterization of geodesic spheres in the hyperbolic space in terms of  a ``weighted'' higher order mean curvature.
 Precisely, we show that a compact  hypersurface $\Sigma^{n-1}$ embedded  in  $\H^n$ with $VH_k$ being constant for some $k=1,\cdots,n-1$ is a centered geodesic sphere. Here $H_k$ is the $k$-th normalized mean curvature of $\Sigma$ induced from $\H^n$ and $V=\cosh r$, where $r$ is a hyperbolic distance to a fixed point in $\H^n$. Moreover, this result can be generalized to  a compact hypersurface $\Sigma$ embedded  in  $\H^n$
  with the ratio $V\left(\frac{H_k}{H_j}\right)\equiv\mbox{constant},\;0\leq j< k\leq n-1$ and $H_j$ not vanishing on $\Sigma$.
\end{abstract}

\title{A new characterization of geodesic spheres in the Hyperbolic space}
\maketitle
\section{Introduction}
A fundamental question about hypersurfaces in differential geometry is the rigidity of the spheres. Alexandrov \cite{A}  studied the rigidity of spheres with constant mean curvature. Precisely, he proved that a compact hypersurface embedded in the Euclidean space with constant mean curvature must be a sphere. This is a well-known Alexandrov theorem. Later, in \cite{Ros0} by exploiting a formula  originated by Reilly \cite{Reilly}, Ros generalized Alexandrov's result to the hypersurface embedded in the Euclidean space with constant scalar curvature. Also, Korevaar \cite{Kor} gave another proof to this result by using the classical reflection method due to Alexandrov \cite{A} and he indicated that Alexandrov's reflexion method works
as well for hypersurfaces in the hyperbolic space and the upper hemisphere.
  Later, Ros \cite{Ros} extended his result to any constant $k$-mean curvature and provided another proof in \cite{MR} together with Montiel. Explicitly,
they proved the following result.
\begin{theo} [\cite{Ros,MR}]\label{thmA}
Let $\Sigma^{n-1}$ be a compact hypersurface
 embedded in Euclidean space $\R^n$. If $H_k$ is constant for some $k=1,\cdots,n-1$,
 then $\Sigma$ is a sphere.
\end{theo}
There are many works generalizing Theorem \ref{thmA}. For instance, In \cite{Koh1,Koh2}, Koh gave a new characterization of spheres in terms of the ratio of two mean curvature which generalized a previous result of Bivens \cite{Bivens}.  Aledo-Al\'as-Romero \cite{AAR} extended the result to compact space-like hypersurfaces with constant higher order mean curvature in de Sitter space.
In \cite{HLMG}, He-Li-Ma-Ge  studied the compact embedded hypersurfaces with constant  higher order anisotropic mean curvatures.  Recently, Brendle \cite{Brendle} showed an Alexandrov type result for hypersurfaces with constant mean curvature in the warped product spaces including de-Sitter-Schwarzschild manifold. Thereafter, Brendle-Eichmair \cite{BE} extended this result to  any star-shaped convex hypersurface with constant higher mean  curvature in the de-Sitter-Schwarzschild manifold.
For the other generalizing works, see for example \cite{AIR,ALM,AM,BC,HMZ,MX,Mon,Rosen} and the references therein.

In a different direction, there is another kind of  mean curvature with a weight $V$ in the Hyperbolic space $\H^n$ which attracts many interests recently. Here $V=\cosh r$, where $r$ is the hyperbolic distance to a fixed point in $\H^n$. For example, the mean curvature integral $\int_{\Sigma}VH_1d\mu$  with weight $V$ appears naturally in the definition of the quasi-local mass in $\H^n$ and the Penrose inequality for asymptotically hyperbolic graphs \cite{DGS}. The higher order mean curvature integrals $\int_{\Sigma}VH_{2k-1}d\mu$ also appear in \cite{GWW} on the GBC mass for asymptotically hyperbolic graphs. Comparing with the one without weight \cite{LWX,GWW1,GWW2,WX}, the corresponding Alexandrov-Fenchel inequalities with weight also hold in the hyperbolic space. Brendle-Hung-Wang \cite{BHW} and de Lima-Gir\~ao \cite{deLG} established  optimal inequalities for the mean curvature integral $\int_{\Sigma}VH_1d\mu$. More recently, in \cite{GWW}, Ge, Wang and the author established an optimal inequality for the higher order mean curvature integral $\int_{\Sigma}VH_kd\mu$. These inequalities are  related to the Penrose inequality for the asymptotically hyperbolic graphs with a horizon type boundary. See \cite{BHW,DGS,deLG,GWW} for instance.

For the hyperbolic space $\H^n$ with the metric  $b=dr^2+\sinh^2r d\Theta^2,$
where $d\Theta^2$ is the standard round metric  on $\mathbb{S}^{n-1}$, the ``weight'' $V=\cosh r$ appears quite naturally. In fact, it satisfies the following equation \begin{equation}\label{Nb}
\mathbb{N}_b\triangleq \{V\in C^{\infty}(\H^n)|\mbox{Hess}^b V=Vb\}.
\end{equation}
Any element $V$ in $\mathbb{N}_b$ satisfies  that the Lorentzian metric $\gamma=-V^2dt^2+b$ is a static solution to the vacuum Einstein equation with the negative cosmological constant $Ric(\gamma)+n\gamma=0.$
In fact, $\mathbb{N}_b$ is an $(n+1)$-dimensional vector space spanned by   an orthonormal basis of the following functions
$$V_{(0)}=\cosh r,\; V_{(1)}=x^1\sinh r,\;\cdots,\; V_{(n)}=x^n\sinh r,$$
where  $x^1,x^2,\cdots, x^n$ are the coordinate functions  restricted to $\mathbb{S}^{n-1}\subset \R^n$. We   equip the vector space $\mathbb{N}_b$  with a Lorentz inner product $\eta$ with signature $(+,-,\cdots,-)$ such that
\begin{equation}\label{eta}
\eta(V_{(0)}, V_{(0)})=1,\qquad \mbox{and}\quad\eta(V_{(i)}, V_{(i)})=-1\quad\mbox{for}\quad i=1,\cdots,n.
\end{equation}
Note that the subset $\mathbb{N}_b^+$ consists of positive functions is just the interior of the future
lightcone. Let $\mathbb{N}_b^1$ be the subset of $\mathbb{N}_b^+$  of functions $V$ with $\eta(V,V)=1$. One can check easily that
every function $V$ in $\mathbb{N}_b^1$ has the following form
\[V=\cosh {\rm dist}_b(x_0, \cdot),\]
for some $x_0\in \H^n$, where ${\rm dist}_b$ is the distance function with respect  to the metric $b$. Therefore,  in the following we fix
 $V=V_{(0)}=\cosh r.$  Throughout this paper, a centered geodesic sphere means a geodesic sphere with the center $x_0$.
 Comparing with the Euclidean case, $V=1$ since the corresponding one to (\ref{Nb}) is 1-dimensional.

It is natural to ask whether the previous rigidity results can be extend to the hypersurface of the hyperbolic space in terms of  some ``weighted'' higher mean curvature. To studying this problem, the existence of function $V$ makes things  subtle,
it would be difficult to apply the classical Alexandrov's reflection method \cite{A} as in \cite{Kor}   to handle this problem. Hence some new ideas may be needed to attack this problem. Comparing with the proof of Theorem \ref{thmA}, the classical Minkowski integral identity (see (\ref{Minkowski_Identity}) below) is not enough in this case. So one needs to generalize the classical Minkowski integral indentity a little bit.  Note that in \cite{GWW}, in order to establish the optimal inequality for the ``weighted'' higher order mean curvature integral, they generalized (\ref{Minkowski_Identity}) into some inequalities between $H_k$ and $H_{k-1}$ for a convex hypersurface in $\H^n$. In fact, by the same proof, this condition can be weakened into a $k$-convex  hypersurface. Another ingredient to our proof is a special case of Heintze-Karcher-type inequality due to \cite{Brendle}. Based on these two aspects, we obtain the following results.
\begin{theo}\label{mainthm1}
Let $\Sigma^{n-1}$ be a compact  hypersurface embedded in the hyperbolic space $\H^n$. If $V H_k$ is constant for some $k=1,\cdots,n-1$, then $\Sigma$ is a centered geodesic sphere.
\end{theo}
The above rigidity result can be generalized to the hypersurface in terms of the ratio of two higher order mean curvatures multiplying by weight.
\begin{theo}\label{mainthm2}
Let $\Sigma^{n-1}$ be a compact hypersurface embedded in the hyperbolic space $\H^n$. If the ratio $V\left(\frac{H_k}{H_j}\right)$ is constant for some $0\leq j< k\leq n-1$ and $H_j$ does not vanish on $\Sigma$, then $\Sigma$ is a centered geodesic sphere.
\end{theo}

\section{Preliminaries}
In this section, first let us  recall some basic definitions and properties of higher order mean curvature.

Let $\s_k$ be the $k$-th elementary symmetry function $\s_k:\R^{n-1}\to \R$ defined by
\[\s_k(\Lambda)=\sum_{i_1<\cdots<i_{k}}\lambda_{i_1}\cdots\lambda_{i_k}\quad  \hbox{ for } \Lambda=(\lambda_1, \cdots,\lambda_{n-1})\in \R^{n-1}.\]
 For a symmetric matrix $B$, denote $\lambda(B)=(\lambda_1(B),\cdots,\lambda_n(B))$ be the eigenvalues of $B$. We set
\[
\s_k(B):=\s_k(\lambda(B)).
\]
 The Garding cone $\Gamma_k$ is defined by
\begin{equation}\label{G-cone}
\Gamma_k=\{\Lambda \in \R^{n-1} \, |\,\s_j(\Lambda)>0, \quad\forall\, j\le k\}.
\end{equation}
A symmetric matrix $B$ is called belong to $\Gamma_k$  if $\lambda(B)\in \Gamma_k$.
Let
\begin{equation}\label{Hk}
H_k=\frac{\s_k}{C_{n-1}^k},
\end{equation}
be the normalized $k$-th elementary symmetry function. As a convention, $H_0=1$. We  have the following Newton-Maclaurin inequalities. For the proof, we refer to a survey of Guan \cite{Guan}.

\begin{lemm} \label{lem} For $1\leq j<k\leq n-1$ and $\Lambda\in \Gamma_k$, we have the following Newton-Maclaurin inequality
\begin{equation}\label{N-M}
H_j\geq (H_k)^{\frac{j}{k}}.
\end{equation}
Moreover,  equality holds in (\ref{N-M}) at $\Lambda$ if and only if $\Lambda=c(1,1,\cdots,1)$.
\end{lemm}

Next, we recall a Minkowski relation proved in \cite{GWW}. In order to state this result, let us give some definitions and notations.
Let $V=cosh(r)$ and $p=\langle DV,\xi\rangle>0$ be the support function. The $k$-th Newton transformation is defined as follows
\begin{equation}\label{Newtondef}
(T_k)^{i}_{j}(B):=\frac{\partial \s_{k+1}}{\partial b^{i}_{j}}(B),
\end{equation}
where $B=(b^{i}_{j})$.  Recall we have the following Minkowski identity with weight $V$ (see \cite{ALM}),
\begin{equation}\label{Minkowski}
\nabla_j(T_{k-1}^{ij}\nabla_i V)=-kp\sigma_{k}+\left(n-k\right) V\sigma_{k-1}.
\end{equation}
Integrating above equality and recalling (\ref{Hk}), we get
\begin{equation}\label{Minkowski_Identity}
\int_{\Sigma} p H_{k}d\mu= \int_{\Sigma}V H_{k-1} d\mu.
\end{equation}
This is the classical Minkowski integral identity in $\H^n$. In \cite{GWW}, (\ref{Minkowski_Identity}) is generalized to some inequalities between $H_k$ and $H_{k-1}$, where they called {\it Minkowski  integral formulae}. These integral inequalities play an important role in this paper.  If the principles of a hypersurface belong to the G\r{a}rding cone $\Gamma_k$, it is called $k$-convex.

\begin{prop}[\cite{GWW}]\label{prop1}
Let $\Sigma^{n-1}$ be a hypersurface isometric immersed in the hyperbolic space $\H^n$, we have
\begin{equation}\label{eq2}
\int_{\Sigma}pVH_kd\mu= \int_{\Sigma}V^2H_{k-1}d\mu+\frac{1}{kC_{n-1}^k}\int_{\Sigma}(T_{k-1})^{ij}\nabla_iV \nabla_j Vd\mu.
\end{equation}
Moreover, if  $\Sigma$ is  $k$-convex, $k\in\{1,\cdots,n-1\}$, then we have
\begin{equation}\label{eq1}
\int_{\Sigma}pVH_kd\mu\geq \int_{\Sigma}V^2H_{k-1}d\mu.
\end{equation}
Equality holds if and only if $\Sigma$ is a centered geodesic sphere in $\H^n$.
\end{prop}
\begin{proof}
This has been proved in \cite{GWW} under the condition that the hypersurface is convex. By the same argument, the result holds for a $k$-convex hypersurface.
 For the convenience of readers, we include the proof.
In view of (\ref{Minkowski}) and (\ref{Hk}), we have
$$\frac{1}{kC_{n-1}^k}\nabla_j(T_{k-1}^{ij}\nabla_i V)=-pH_k+H_{k-1}V.$$
Multiplying above  equation by the function $V$ and integrating by parts, one obtains the desired result (\ref{eq2}). Under the assumption that $\Sigma$ is  $k$-convex,  the $(k-1)$-th Newton tensor  $T_{k-1}$ is positively definite (the proof see the one of Proposition $3.2$ in \cite{BC} for instance), that is,
$$(T_{k-1})^{ij}\nabla_i V \nabla_j V\geq 0,$$
thus (\ref{eq1}) holds. When the equality holds, we have $\nabla V=0$ which implies that $\Sigma$ is a geodesic sphere in $\H^n.$
\end{proof}

Finally, we need  a special case of the Heintze-Karcher-type
inequality due to Brendle \cite {Brendle}.
\begin{prop}[Brendle]\label{B}
Let $\Sigma$ be a compact hypersurface embedded in $\H^n$ with positive mean curvature (i.e. $H_1>0$), then
$$\int_{\Sigma}pd\mu\leq \int_{\Sigma} \frac{V}{H_1}d\mu.$$
Moreover, equality holds if and only if $\Sigma$ is totally umbilical.
\end{prop}
The idea to prove this inequality is first to choose a suitable geometric flow, then prove the monotonicity, finally analyze the asymptotical behavior to obtain the desired inequality.
This Heintze-Karcher type
inequality plays an important role in \cite{deLG} to establish a weighted Alexandrov-Fenchel equality for mean curvature integral in the hyperbolic space. This inequality is also a main  ingredient  for this  paper.
For the proof of Proposition \ref{B}, we refer the readers to \cite{Brendle}.

\section{Proof of the main theorems}
With all the preparing work, now we are ready to prove our main theorems.
\vspace{3mm}
\

\noindent {\it Proof of  Theorem \ref{mainthm1}.}
Since the hypersurface $\Sigma$ is compact,  thus at the point where the distance function $r$ of $\H^n$ attains its maximum, all the principal curvatures are positive. This together with the simple fact $V=\cosh r>0$ yields that $VH_k$ is a positive constant and thus $H_k$ is  positive everywhere in $\Sigma$. From the result of G\r{a}rding \cite{Garding}, we know that  the principal curvatures of $\Sigma$ belong to the Garding cone $\Gamma_k$ defined in (\ref{G-cone}).

It follows from (\ref{eq1}) that
\begin{equation}\label{eq3}
 V H_k\int_{\Sigma}pd\mu=\int_{\Sigma} pVH_kd\mu\geq \int_{\Sigma}V^2H_{k-1}d\mu.
\end{equation}
By the Newton-Maclaurin inequality (\ref{N-M}), we have
$$H_{k-1}\geq H_k^{\frac{k-1}{k}},$$
thus
\begin{equation}\label{eq4}
\int_{\Sigma}V^2H_{k-1}d\mu\geq \int_{\Sigma} V^2 H_k^{\frac{k-1}{k}}d\mu=(VH_k)^{\frac{k-1}{k}}\int_{\Sigma}V^{1+\frac 1k}d\mu.
\end{equation}
Hence (\ref{eq3}) and (\ref{eq4}) yield
\begin{equation}\label{eq5}
\int_{\Sigma}pd\mu\geq (VH_k)^{-\frac 1k}\int_{\Sigma}V^{1+\frac 1k}d\mu,
\end{equation}
and  equality holds if and only $\Sigma$ is a  geodesic sphere.
On the other hand, applying Proposition \ref{B} and the Newton-Maclaurin inequality (\ref{N-M}) we derive that
\begin{equation}\label{eq6}
\int_{\Sigma}pd\mu\leq \int_{\Sigma} \frac{V}{H_1}d\mu\leq \int_{\Sigma}\frac{V}{H_k^{\frac 1k}}d\mu=(VH_k)^{-\frac 1k}\int_{\Sigma} V^{1+\frac 1k}d\mu.
\end{equation}
Finally combining (\ref{eq5}) and (\ref{eq6}) together, we complete the proof.
\qed

\vspace{8mm}
\noindent {\it Proof of  Theorem \ref{mainthm2}.}
The first step is more or less the same as above.
At the point where the distance function $r$ of $\H^n$ attains its maximum, all the principal curvatures are positive. Therefore $H_j$ and $H_k$ are both positive at the point. This together with  $V=\cosh r>0$ yields that $V\left(\frac{H_k}{H_j}\right)$ is a positive constant. Since by the assumption that $H_j$ does not vanish on $\Sigma$ , then $H_j$ and $H_k$ are  positive everywhere in $\Sigma$. From  \cite{Garding}, we know that the principal curvatures of $\Sigma$ belong to $\Gamma_k$ defined in (\ref{G-cone}).

If $j=0$, Theorem \ref{mainthm2} is reduced to the case of Theorem \ref{mainthm1}. In the following, we consider the case $j\geq1.$
Denote the positive constant by $\alpha$, namely, $$\alpha:=V\left(\frac{H_k}{H_j}\right)>0.$$
Applying the Newton-Maclaurin inequality (\ref{N-M}), we note that
$$\frac{H_k}{H_{k-1}}\leq\frac{H_j}{H_{j-1}},$$
which implies
\begin{equation}\label{eq1_thm2}
V\left(\frac{H_{k-1}}{H_{j-1}}\right)\geq\alpha.
\end{equation}
It follows from (\ref{eq1}) and (\ref{Minkowski_Identity}) that
\begin{equation}\label{eq2_thm2}
\int_{\Sigma}V^2H_{k-1}d\mu\leq\int_{\Sigma} VpH_kd\mu=\alpha\int_{\Sigma}pH_jd\mu=\alpha\int_{\Sigma}VH_{j-1}d\mu.
\end{equation}
This gives
$$\int_{\Sigma}V(VH_{k-1}-\alpha H_{j-1})\leq 0.$$
Above together with (\ref{eq1_thm2}) imply
$$V\frac{H_{k-1}}{H_{j-1}}=\alpha,$$
everywhere in $\Sigma.$
By an iteration argument, one obtains
$$V\frac{H_{k-j}}{H_0}=V H_{k-j}=\alpha,$$
everywhere in $\Sigma.$
Finally, from Theorem \ref{mainthm1}, we complete the proof.
\qed
\vspace{5mm}

\noindent{\it Acknowledgment.} The author would like to thank Professor Guofang Wang for suggesting this problem and useful discussions.

\end{document}